\documentclass[11pt]{article}
\usepackage[utf8]{inputenc}
\usepackage[english]{babel}
\usepackage{amsmath}
\usepackage{amssymb}
\usepackage{amsthm}
\usepackage{verbatim}
\usepackage{graphicx}
\usepackage{subcaption}
\usepackage[shortlabels]{enumitem}
\usepackage{xcolor}
\usepackage{float}

\usepackage[pagewise]{lineno}

\newtheorem{theorem}{Theorem}[section]
\newtheorem{corollary}{Corollary}[theorem]
\newtheorem{lemma}[theorem]{Lemma}

\providecommand{\keywords}[1]{{\textit{Keywords:}} #1}

\newcommand{\MM}{{\mathsf{M}}}
\newcommand{\LL}{{\mathsf{L}}}
\newcommand{\MMd}{{\mathsf{M}_\partial}}
\newcommand{\ff}{{\mathsf{f}}}

\newcommand{\Se}{{S_\epsilon}}

\newcommand{\NKT}{\mathcal{N}(K)^\bot}

\newcommand{\Npde}{V_{00}}

\newcommand{\KK}{{\mathsf{K_h}}}
\newcommand{\LKK}{{\mathsf{K}}}
\newcommand{\WW}{{\mathsf{W}}}
\newcommand{\PP}{{\mathsf{P}}}
\newcommand{\II}{{\mathsf{I}}}
\newcommand{\xx}{{\mathbf{x}}}
\newcommand{\yy}{{\mathbf{y}}}
\newcommand{\zz}{{\mathbf{z}}}
\newcommand{\bb}{{\mathbf{b}}}

\newcommand{\nn}{{\mathbf{n}}}

\newcommand{\NK}{\mathcal{N}(\LKK)}

\newcommand{\pphi}{{\mathbf{\phi}}}
\newcommand{\hphi}{{\hat{\phi}}}
\newcommand{\hhphi}{{\mathbf{\hphi}}}

\newcommand{\nullspace}[1]{{\mathcal{N}(#1)}}

\DeclareMathOperator*{\argmin}{arg\,min}
\DeclareMathOperator*{\argmax}{arg\,max}

\title{A regularization operator for source identification for elliptic PDEs} 
\author{Ole L{\o}seth Elvetun\thanks{Faculty of Science and Technology, Norwegian University of Life Sciences, P.O. Box 5003, NO-1432 {\AA}s, Norway. Email: ole.elvetun@nmbu.no.} \, and Bj{\o}rn Fredrik Nielsen\thanks{Faculty of Science and Technology, Norwegian University of Life Sciences, P.O. Box 5003, NO-1432 {\AA}s, Norway. Email: bjorn.f.nielsen@nmbu.no. Nielsen's work was supported by The Research Council of Norway, project number 239070.} }

\begin{document}

\maketitle

\begin{abstract} 
We study a source identification problem for a prototypical elliptic PDE from Dirichlet boundary data. This problem is ill-posed, and the involved forward operator has a significant nullspace. 
Standard Tikhonov regularization yields solutions which approach the minimum $L^2$-norm least-squares solution as the regularization parameter tends to zero. We show that this approach 'always' suggests that the unknown local source is very close to  the boundary of the domain of the PDE, regardless of the position of the true local source. 

We propose an alternative regularization procedure, realized in terms of a novel  regularization operator, which is better suited for identifying local sources positioned anywhere in the domain of the PDE. Our approach is motivated by the classical theory for Tikhonov regularization and yields a standard quadratic optimization problem. Since the new methodology is derived for an abstract operator equation, it can be applied to many other source identification problems. 
This paper contains several numerical experiments and an analysis of the new methodology.
\end{abstract}
\keywords{Inverse source problems, PDE-constrained optimization, Tikhonov regularization}

\section{Introduction}
\label{section:introduction}
We will study the problem of identifying the source in a prototypical elliptic PDE from Dirichlet  boundary data: 
    \begin{equation} \label{eq1}
        \min_{(f,u) \in F_h \times H^1(\Omega)} \left\{ \frac{1}{2}\|u-d\|_{L^2(\partial\Omega)}^2 + \frac{1}{2}\alpha\|\WW f\|_{L^2(\Omega)}^2 \right\}
    \end{equation}
    subject to 
    \begin{equation} \label{eq2}
    \begin{split}
        -\Delta u + \epsilon u &= f \quad \mbox{in } \Omega, \\
        \frac{\partial u}{\partial \nn} &= 0  \quad \mbox{on } \partial \Omega, 
    \end{split}
    \end{equation}
    where $F_h$ is a finite dimensional subspace of $L^2(\Omega)$, $\WW: F_h \rightarrow F_h$ is a regularization operator, $\alpha > 0$ is a regularization parameter, \textcolor{black}{$d$ is boundary data}, $\epsilon$ is a positive parameter, $\nn$ denotes the outwards pointing unit normal vector of the boundary $\partial \Omega$ of the bounded domain $\Omega$, and $f$ is the unknown source. \textcolor{black}{(In other words, we attempt to use the Dirichlet boundary data $u=d$ on $\partial \Omega$, applying the formulation \eqref{eq1}-\eqref{eq2}, to identify $f$.)}

This problem, and variants of it, appear in many applications. For example, in crack determination \cite{Alves04}, in EEG \cite{bail01, elul72} and in the inverse ECG problem  \cite{Nie13a,Wan13}. 



Even though most source identification tasks for elliptic PDEs are ill-posed, several methods for computing reliable results have been developed. Typically, one assumes a priori that $f$ is composed of a finite number of pointwise sources or sources having compact support within a small number of finite subdomains, see, e.g., \cite{Abdel15,babda09,Bad00,Han11,Zha19} and references therein. Such approaches lead to involved mathematical issues, but in many cases optimization procedures and/or explicit regularization can be avoided. Furthermore, some of these "direct methods" can also recover more general sources \cite{Wan17,Zha18} when $\epsilon < 0$, i.e.,  for the Helmholtz equation with multi-frequency data.   


As an alternative to searching for point sources, the authors of \cite{cheng15,song12} restrict the control domain to a subdomain and introduce a Kohn-Vogelius fidelity term. In \cite{hinze19} the authors also use a Kohn-Vogelius functional, but instead of restricting the control domain, they search for the source term closest to a given prior. 

The related problem of determining the interface between two regions with constant densities (sources) has also been studied \cite{kun94, ring95}. More specifically, in these investigations $f$ has the form $f(x)=\rho_1, \, \rho_2$ in $\Omega_1, \, \Omega_2$, respectively, and one seeks to identify the subdomains $\Omega_1$ and $\Omega_2$. 
Here, $\rho_1$ and $\rho_2$ are given constants. Moreover, since the 1990s very sophisticated analyses have been undertaken in order to further determine information about the support of the source term from boundary data, see, e.g., \cite{Han11,Het96,BIsa05}. 



In this paper we will not make any assumptions about the form of the control $f$, nor restrict the control domain. Instead we introduce a weighting $\WW$, also referred to as a regularization operator, in the regularization term. This enables us to locate a single local source positioned anywhere in the domain without any prior knowledge about its position. More specifically, we will show that, if the true source equals any of the basis functions used to discretize the control, then the inverse solution will be closer to the true source, in $L^2$-sense, than a \textcolor{black}{particular} function which achieves its maximum at the same location as the true source. \textcolor{black}{This particular function can be derived from the outcome of applying standard Tikhonov regularization.}
Numerical experiments indicate that our scheme also can identify several well-separated and isolated (local) sources, but we do not have a rigorous mathematical analysis covering such cases. Our approach leads to a standard quadratic optimization problem. 

The projection onto the orthogonal complement of the nullspace of the forward operator 
$$
\KK: F_h \rightarrow L^2(\partial \Omega), \quad
f \mapsto u|_{\partial \Omega},$$ 
associated with \eqref{eq1}-\eqref{eq2}, plays an important role in our study. More precisely, 
our regularization operator $\WW$ can be interpreted as a scaling of the basis functions for $F_h$.  This scaling is such that the lengths of the projections of the modified basis functions, onto the orthogonal complement of the nullspace of $\KK$, is the same. 

This investigation is further motivated in section \ref{section:motivation}, and our regularization operator is derived in section \ref{section:regularization_operator}, which also clarifies why we use a finite dimensional space $F_h$ for the control $f$ in \eqref{eq1}-\eqref{eq2}. Section \ref{section:analysis} is devoted to an analysis of the new methodology. The numerical experiments are presented in section \ref{section:numerical_experiments}, and section \ref{section:discussion} contains a brief summary and an open problem.   

\section{Motivation} 
\label{section:motivation}
The right panel in Figure \ref{fig:square_ST} shows the numerical solution $f_h$ of \eqref{eq1}-\eqref{eq2} when the true source $f_{\mathrm{true}}$ is as depicted in the left panel. In these computations we employed standard Tikhonov regularization, i.e., $\WW=\II$ and \textcolor{black}{$\epsilon=\alpha=10^{-3}$}. More specifically, we solved \eqref{eq1}-\eqref{eq2} numerically with $d = u_{\mathrm{true}}|_{\partial \Omega}$, where $u_{\mathrm{true}}$ denotes the numerical solution of the boundary value problem \eqref{eq2} with $f=f_{\mathrm{true}}$. We observe that, even in the noise free situation, we can not recover the position of the true source when standard Tikhonov regularization is used, and the computed source $f_h$ is mainly located at the boundary $\partial  \Omega$ of $\Omega$, even though the true source $f_{\mathrm{true}}$ has it support in the interior of $\Omega$. The mathematical explanation for this is as follows. 

\begin{figure}
    \centering
    \begin{subfigure}[b]{0.48\linewidth}        
        \centering
        \includegraphics[width=\linewidth]{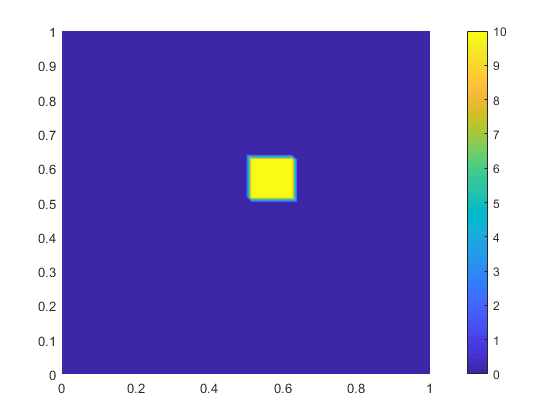}
        \caption{True source}
        \label{fig:square_ST_a}
    \end{subfigure}
    \begin{subfigure}[b]{0.48\linewidth}        
        \centering
        \includegraphics[width=\linewidth]{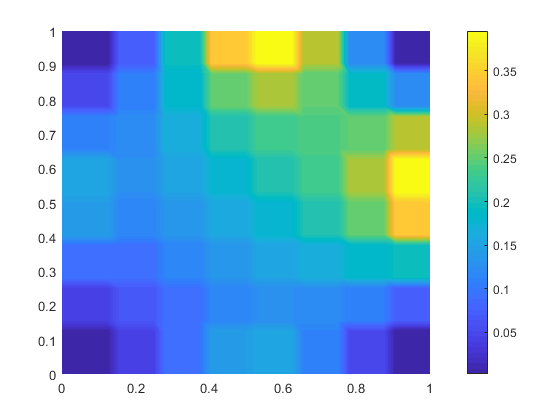}
        \caption{Inverse solution}
    \end{subfigure}
    \caption{Comparison of the true source and the inverse solution using standard Tikhonov regularization with $\alpha = 10^{-3}$.}
    \label{fig:square_ST}
\end{figure}

Consider the following 
(continuous) $L^2$-version of our source identification problem with standard Tikhonov regularization: 
\begin{equation}
    \label{eq3}
    \min_{(f,u) \in L^2(\Omega) \times H^1(\Omega)} \left\{ \frac{1}{2}\|u-d\|_{L^2(\partial\Omega)}^2 + \frac{1}{2}\alpha\| f\|_{L^2(\Omega)}^2 \right\}
\end{equation}
subject to \eqref{eq2}. Throughout this paper we assume that $\Omega$ is a bounded domain with a piecewise smooth boundary $\partial \Omega$. 

We denote the mapping $f \mapsto u|_{\partial\Omega}$, associated with \textcolor{black}{\eqref{eq2} and \eqref{eq3}}, 
by the forward operator 
\begin{equation}
  \LKK: L^2(\Omega) \rightarrow L^2(\partial\Omega). \nonumber
\end{equation} 
More precisely, $\LKK f = u|_{\partial\Omega}$, where $u$ is the unique (weak) solution of the boundary value problem \eqref{eq2}. 

The nullspace $\NK$ of $\LKK$ consists of functions $f \in L^2(\Omega)$ which yield solutions of \eqref{eq2} with
zero trace on $\partial\Omega$. If we define the space $\Npde$ as
\begin{equation}
    \Npde = \left\{\psi \in C^2(\overline{\Omega}): \, \psi = \frac{\partial\psi}{\partial\mathbf{n}} = 0 \textnormal{ on } \partial\Omega \right\}, \nonumber
\end{equation}
we observe that\footnote{If $\psi \in C^2(\overline{\Omega})$ and $\Omega$ is bounded, then $\Delta \psi \in L^2(\Omega)$.} 
\begin{equation}
   Q = \left\{q = -\Delta \psi + \epsilon \psi, \, \psi \in \Npde\right\} \subset \NK. \label{eq:NK} 
\end{equation}

Let $(f_{\alpha}^*,u_{\alpha}^*)$ denote the solution of \textcolor{black}{\eqref{eq2} and \eqref{eq3}}, $\alpha>0$, and assume that the limit 
\[
\lim_{\alpha \rightarrow 0} f_{\alpha}^* = f^* = \LKK^\dagger d 
\]
is a $C^2$-function, i.e., $f^* \in C^2(\overline{\Omega})$. 
Here, $\LKK^{\dagger}$ denotes the Moore-Penrose inverse of $\LKK$.
From standard theory we know that the minimum norm least-squares solution $f^*$ belongs to the orthogonal complement of the nullspace of $\LKK$, i.e., $f^* \in \NKT$, or  
\begin{equation}
    (f^*, q)_{L^2(\Omega)} = 0 \quad \forall q \in Q \subset \NK, \nonumber
\end{equation}
which implies that
\begin{equation}
    (f^*, -\Delta \psi + \epsilon \psi)_{L^2(\Omega)} = 0 \quad \forall \psi \in \Npde. \nonumber
\end{equation}
Invoking integration by parts/Green's formula yields that 
\begin{equation}
    (-\Delta f^* + \epsilon f^*,\psi)_{L^2(\Omega)} = 0 \quad \forall \psi \in \Npde, \nonumber
\end{equation}
and we can conclude that $$-\Delta f^* + \epsilon f^* = 0 \quad \mbox{in } \Omega.$$ 

Standard maximum principles for elliptic PDEs thus assure that $f^*$ cannot attain a non-negative maximum\footnote{If $\epsilon=0$, then $f^*$ will achieve its maximum on the boundary $\partial \Omega$.} in the interior of $\Omega$, see, e.g., Theorem 4.10 in \cite{BRen93}. For small $\alpha>0$, $f_{\alpha}^* \approx f^*$, and the use of ordinary Tikhonov regularization will therefore fail to identify internal sources. This explains the results reported in Figure \ref{fig:square_ST}. 

The paper \cite{Bad98} contains results related to the analysis presented in this section: For $\epsilon=\alpha=0$, \cite{Bad98} clarifies the role of harmonic sources. \textcolor{black}{Also note that the argument presented above does not hold when $\epsilon < 0$, i.e., it can not be applied to problems involving the Helmholtz equation (because maximums principles are not readily available).}

\section{The regularization operator}
\label{section:regularization_operator} 
Motivated by the results presented above, we will now construct a regularization operator $\WW$ better suited for recovering internal sources. However, for the sake of generality, we proceed by considering the following abstract operator equation:   
\begin{equation}
\label{A1}
\KK \xx = \bb, 
\end{equation}
where $\KK: X \rightarrow Y$ is a linear operator with a {\em nontrivial nullspace} and possibly very small singular values. The real vector spaces $X$ and $Y$ are finite dimensional and $\bb \in Y$. 

Employing modified Tikhonov regularization yields the problem 
\begin{equation}
    \label{A2}
\zz_{\alpha} = \argmin_{\zz} \left\{ \frac{1}{2} \| \KK \zz - \bb \|_Y^2 + \frac{1}{2} \alpha \| \WW \zz \|_X^2 \right\}, 
\end{equation}
where $\alpha >0$ is a regularization parameter, $\WW:X \rightarrow X$ is an invertible regularization operator, and $\| \cdot \|_X$ and $\| \cdot \|_Y$ denote the norms of $X$ and $Y$, respectively. 

By defining 
\[
\yy_{\alpha}=\WW \zz_{\alpha}, 
\]
we get the problem  
\begin{equation}
 \label{A2.1}
\yy_{\alpha} = \argmin_{\yy} \left\{ \frac{1}{2} \| \KK \WW^{-1} \yy - \bb \|_Y^2 + \frac{1}{2} \alpha \| \yy \|_X^2 \right\}, 
\end{equation}
and according to standard theory for Tikhonov regularization, see, e.g., \cite{BEng96}, 
\[
\lim_{\alpha \rightarrow 0} \yy_{\alpha} = \yy^* = (\KK \WW^{-1})^{\dagger} \bb. 
\] 
Hence, the introduction of $\WW$ in the regularization term in \eqref{A2} motivates us to consider the related equation 
\begin{equation}
\label{A3}
\KK \WW^{-1} \yy = \bb. 
\end{equation}
We will now use \eqref{A3} to motivate a particular choice of a regularization operator $\WW$. 

Let $\pphi_1, \, \pphi_2, \, \ldots, \pphi_n$ be a basis for $X$, i.e., 
\[
X=\mathrm{span} \{ \pphi_1, \, \pphi_2, \, \ldots, \pphi_n \}.
\]
The minimum norm least-squares solution
\[
\xx^* = \sum_{i=1}^n x_i^* \pphi_i = \KK^{\dagger} \bb 
\]
of \eqref{A1} belongs to the orthogonal complement of the nullspace of $\KK$, i.e.,  
\[
\xx^* \in \nullspace{\KK}^{\perp}. 
\]
Hence, if 
\begin{equation}
\label{A3.1}
\PP: X \rightarrow \nullspace{\KK}^{\perp}
\end{equation}
denotes the orthogonal projection, then 
\begin{align*}
    \xx^* &= \PP \xx^* \\
    &= \sum_{i=1}^n x_i^* \, \PP \pphi_i. 
\end{align*}
Note that the norm $\| \PP \pphi_i \|_X$ of $\PP \pphi_i$ depends on the angle between $\pphi_i$ and $\nullspace{\KK}^{\perp}$. Roughly speaking, the minimum norm least-squares solution $\xx^*$ will typically be dominated by the basis vectors which has a relatively small angle to $\nullspace{\KK}^{\perp}$ -- one may say that the basis $\pphi_1, \, \pphi_2, \, \ldots, \pphi_n$ is biased because these functions' contributions to $\xx^*$ are depending on the norms of their projections onto the orthogonal complement of the nullspace of $\KK$. 

Let us now assume that 
\[
\| \PP \phi_i \|_X \neq 0, \, i=1,2,\ldots,n, 
\]
and note that the scaled basis 
\[
\hphi_i = \frac{\phi_i}{\| \PP \pphi_i \|_X}, \, i=1,2,\ldots,n,
\]
has the property 
\begin{equation}
\label{A3.9}
\| \PP \hphi_i \|_X = 1, \, i=1,2,\ldots,n. 
\end{equation}
Consider equation \eqref{A3}, where 
\[
\yy =\sum_{i=1}^n y_i \pphi_i. 
\]
Motivated by the previous paragraph, we want to choose the regularization operator $\WW$ such that the projection $\PP (\WW^{-1} \yy)$ of $\WW^{-1} \yy$ onto $\nullspace{\KK}^{\perp}$ is a sum of the components $y_1, y_2, \ldots, y_n$ of $\yy$ times vectors which have equal length. This is accomplished as follows:  
Provided that the linear regularization operator $\WW: X \rightarrow X$ is defined by 
\begin{equation}
    \label{A4}
    \WW \pphi_i = \| \PP \pphi_i \|_X \pphi_i, \quad i=1,2,\ldots, n, 
\end{equation}
we find that 
\begin{align*}
    \PP (\WW^{-1} \yy) &= \sum_{i=1}^n y_i \PP (\WW^{-1} \pphi_i) \\
    &= \sum_{i=1}^n y_i \PP (\| \PP \pphi_i \|_X^{-1} \pphi_i) \\ 
    &= \sum_{i=1}^n y_i \PP \hhphi_i, 
\end{align*}
where all the involved projections $\PP \hhphi_1, \, \PP \hhphi_2, \, \ldots, \PP \hhphi_n$ have length one  \eqref{A3.9}. (Appendix \ref{alternative-motivation} contains an alternative motivation for the definition \eqref{A4} of the regularization operator $\WW$.)

The discussion presented above can be generalized to separable Hilbert spaces, assuming that none of the basis functions $\{ \phi_1, \, \phi_2, \, \ldots \}$ belong to the nullspace of $\LKK$. In fact, in order to obtain a 'reasonable' regularization operator $\WW$, as defined in \eqref{A4}, one should make sure that $\min_i \{ \| \PP \phi_i \|_X \}$ does not become too small, relative to the noise level in $\bb$: If $\| \PP \phi_i \|_X$ is very small, one tries to recover the contribution associated with a basis function $\phi_i$ which is almost in the nullspace of the forward operator. Hence, in order to not get 'too close' to the nullspace, one would typically use a rather moderate number $n$ of basis functions, and these basis functions should have a relatively significant support. This is our motivation for employing a finite dimensional space for the control $f$ in \eqref{eq1}-\eqref{eq2}.  


Let $\WW$ be as defined in \eqref{A4}. In the next sections we will explore the following three methods for identifying sources: 
\begin{description}
\item[Method I:] We compute 
\[
\WW^{-1} \xx_{\alpha}, 
\]
where 
\[
\xx_{\alpha} = \argmin_{\xx} \left\{ \frac{1}{2} \| \KK \xx - \bb \|_Y^2 + \frac{1}{2} \alpha \| \xx \|_X^2 \right\}, 
\]
i.e., $\xx_{\alpha}$ is the outcome of standard Tikhonov regularization. 
\item[Method II:] We compute
\[
\yy_{\alpha} = \argmin_{\yy} \left\{ \frac{1}{2} \| \KK \WW^{-1} \yy - \bb \|_Y^2 + \frac{1}{2} \alpha \| \yy \|_X^2 \right\}. 
\]
\item[Method III:] We compute
\[
\zz_{\alpha} = \argmin_{\zz} \left\{ \frac{1}{2} \| \KK \zz - \bb \|_Y^2 + \frac{1}{2} \alpha \| \WW \zz \|_X^2 \right\}, 
\]
i.e., $\zz_{\alpha} = \WW^{-1} \yy_{\alpha}$. 
\end{description}
It turns out that these methods yield rather similar visual results for recovering a single well-localized source: see Figure \ref{fig:square_new} for the recovery of the true source depicted in panel (a) in Figure \ref{fig:square_ST}.
Methods II and III work better for identifying several local sources than Method I. This will be exemplified in the numerical experiments section. 

As we will see in the next section, our analyses of methods II and III rely on the investigation of Method I. 

\section{Analysis}
\label{section:analysis}
We will now investigate whether Method I, Method II and Method III
can recover the individual basis functions $\pphi_1, \, \pphi_2, \, \ldots,\pphi_n$. More precisely, if the right-hand-side $\bb$ in \eqref{A1} equals the image of $\pphi_j$ under $\KK$, i.e., $$\bb=\KK \pphi_j,$$ can we employ these methods to roughly recover the 'position' of $\pphi_j$? 
When standard Tikhonov regularization is used, i.e., when $\WW = \II$, the numerical experiments and the analysis presented in section \ref{section:motivation} show that this is not necessarily the case. Ideally, we would like to analyze the recovery of vectors in a large subset of $X$ from their images in $Y$ under $\KK$, but we have not been able to do so.

For the sake of simplicity, we consider the limit case $\alpha \rightarrow 0$ in this section. Our analysis thus address some mathematical properties of the minimum norm least-squares solutions of the linear problems associated with methods I, II and III.    

The simple result presented in our first lemma is not explicitly formulated in standard texts. For the sake of completeness, and since it will be used below, we now prove: 
\begin{lemma}\label{lem:easyEq}
   Let $\mathcal{A}: H_1 \rightarrow H_2$ be a bounded linear operator, where $H_1$ and $H_2$ are Hilbert spaces. 
   For any $\psi \in H_1$, the minimum norm least-squares solution of 
   \begin{equation}
   \label{easyEq}
       \mathcal{A}u = \mathcal{A} \psi   
   \end{equation}
   is $$u^* = \mathcal{P}\psi,$$ where $\mathcal{P}: H_1 \rightarrow \nullspace{\mathcal{A}}^\bot$ denotes the  orthogonal projection of elements in $H_1$ onto the orthogonal complement of the nullspace of $\mathcal{A}$. 
\end{lemma}
\begin{proof} 
We observe that 
\begin{align*} 
\mathcal{A} \mathcal{P}\psi &= \mathcal{A} ((\mathcal{P}\psi-\psi)+\psi) = \mathcal{A} \psi
\end{align*}
because $(\mathcal{P}\psi-\psi) \in \nullspace{\mathcal{A}}$. Hence, $\mathcal{P}\psi$ is a solution of \eqref{easyEq}. Since $\mathcal{A}$ is linear, any other solution of \eqref{easyEq} can be written in the form 
$\mathcal{P}\psi + \tau$, for some $\tau \in \nullspace{\mathcal{A}}$, which has norm
\[
\| \mathcal{P}\psi + \tau \|_{H_1} = \sqrt{\| \mathcal{P}\psi \|_{H_1}^2 + \| \tau \|_{H_1}^2} \geq \| \mathcal{P}\psi \|_{H_1}.
\]
\end{proof}

We will now prove that Method I can recover the individual basis functions in the sense that $\WW^{-1} \xx^*$ attains its maximum at the correct position. More precisely, $\WW^{-1} \xx^*$ attains its maximum for the correct index.   
\begin{theorem} {\bf (Method I).}
Let $\WW$ be the regularization operator defined in \eqref{A4} and assume the that the basis $\mathcal{B}=\{ \pphi_1, \pphi_2, \ldots, \pphi_n\}$ is orthonormal. Then, for any $j \in \{1,2, \ldots, n\}$, the minimum norm least-squares solution $\xx_j^*$ of 
\begin{equation}
\label{recoveryEq}
\KK \xx = \KK \pphi_j
\end{equation}
satisfies 
\[
\WW^{-1} \xx_j^* = \| \PP \phi_j \|_X \sum_{i=1}^{n} \left( \frac{\PP \phi_j}{\| \PP \phi_j \|_X}, \frac{\PP \pphi_i}{\| \PP \phi_i \|_X} \right)_X  \pphi_i,  
\]
where $\PP: X \rightarrow \nullspace{\KK}^{\perp}$ denotes the orthogonal projection of elements in $X$ onto the orthogonal complement of the nullspace of $\KK$. 
Hence, 
\[
j \in \argmax_{i \in \{1,2, \ldots, n\}} \left( \WW^{-1} \xx_j^* (i) \right),
\]
where $\WW^{-1} \xx_j^* (i)$ denotes the $i$'th component of the vector $[\WW^{-1} \xx_j^*]_{\mathcal{B}} \in \mathbb{R}^n$.
\end{theorem}
\begin{proof}
  Invoking Lemma \ref{lem:easyEq}, the assumption that the basis is orthonormal, the definition \eqref{A4} of $\WW$ and basic properties of projections it follows that 
  \begin{align}
  \nonumber
      \WW^{-1} \xx_j^* &= \WW^{-1} \PP \pphi_j \\ 
      \nonumber
      &= \WW^{-1} \sum_{i=1}^{n} (\PP \phi_j,\pphi_i)_X \pphi_i \\ 
      \nonumber
      &= \sum_{i=1}^{n} (\PP \phi_j,\pphi_i)_X \| \PP \phi_i \|_X^{-1} \pphi_i \\
      \nonumber
      &= \sum_{i=1}^{n} (\PP \phi_j, \PP \pphi_i)_X \| \PP \phi_i \|_X^{-1} \pphi_i \\
      \label{basicEq}
      &=\| \PP \phi_j \|_X \sum_{i=1}^{n} \left( \frac{\PP \phi_j}{\| \PP \phi_j \|_X}, \frac{\PP \pphi_i}{\| \PP \phi_i \|_X} \right)_X  \pphi_i.
  \end{align}
\end{proof}

Method II involves the operator $\KK \WW^{-1}$. In the argument presented below we use the orthogonal projection $\tilde{\PP}$ onto the orthogonal complement of the nullspace of $\KK \WW^{-1}$, 
\[
\tilde{\PP}: X \rightarrow \nullspace{\KK \WW^{-1}}^{\perp}.
\]

We will now prove that a scaled version of Method II yields a solution which, in norm sense, is better than the outcome of Method I. 
\begin{theorem} {\bf (Method II).} 
\label{theorem:minnorm-y}
  Assume that $\{ \pphi_1, \, \pphi_2, \, \ldots, \pphi_n \}$ is an {\em orthonormal  basis} and let $\WW$ be the operator defined in \eqref{A4}. Then the minimum norm least-squares solution $\yy_j^*$ of 
  \begin{equation}
\label{BB1}
\KK \WW^{-1} \yy = \KK \pphi_j, 
\end{equation}
satisfies 
\begin{align*}
    \left\| \pphi_j - \frac{\yy_j^*}{\| \PP \pphi_j \|_X} \right\|_X \leq  \| \pphi_j - \WW^{-1} \xx_j^* \|_X,  
\end{align*}
where $\xx_j^*$ is the minimum norm least-squares solution of \eqref{recoveryEq} and $\WW^{-1} \xx_j^*$ can be written in the form \eqref{basicEq}. 
\end{theorem}
\begin{proof}
  The minimum norm least-squares solution of the auxiliary problem 
\begin{equation}
\label{BB2}
\KK \WW^{-1} \hat{\yy} = \KK \WW^{-1}  \pphi_j \, (= \| \PP \pphi_j \|_X^{-1} \KK \pphi_j) 
\end{equation}
is, according to Lemma \ref{lem:easyEq}, 
\[
\hat{\yy}_j^* = \tilde{\PP} \pphi_j.  
\] 
Also, since $\tilde{\PP}$ is the orthogonal projection onto $\nullspace{\KK \WW^{-1}}^{\perp}$, 
\[
\| \pphi_j - \hat{\yy}_j^* \|_X \leq \| \pphi_j - r \|_X \, \mbox{for all } r \in \nullspace{\KK \WW^{-1}}^{\perp}.
\]

From \eqref{BB2} we find that 
\[
\KK \WW^{-1} (\| \PP \pphi_j \|_X \hat{\yy}) = \KK \pphi_j,
\]
and therefore there is as simple connection between the minimum norm least-squares solutions $\yy_j^*$ and $\hat{\yy}_j^*$ of \eqref{BB1} and \eqref{BB2}, respectively: 
\[
\yy_j^* = \| \PP \pphi_j \|_X \hat{\yy}_j^*. 
\]
We thus conclude that 
\begin{equation}
\label{BB3}
\left\| \pphi_j - \frac{\yy_j^*}{\| \PP \pphi_j \|_X} \right\|_X \leq \| \pphi_j - r \|_X \, \mbox{for all } r \in \nullspace{\KK \WW^{-1}}^{\perp}.
\end{equation}

Observe that 
\[
q \in \nullspace{\KK} \iff \WW q \in \nullspace{\KK \WW^{-1}}. 
\]
We know that 
\[
\xx_j^* \in \nullspace{\KK}^{\perp}, 
\]
or 
\[
(\xx_j^*, q)_X = 0 \quad \mbox{for all } q \in \nullspace{\KK}. 
\]
The operator $\WW$ is self-adjoint and hence 
\[
(\WW^{-1}\xx_j^*, \WW q)_X = 0 \quad \mbox{for all } q \in \nullspace{\KK}. 
\]
We can thus conclude that $\WW^{-1}\xx_j^* \in \nullspace{\KK \WW^{-1}}^{\perp}$, and the result follows from \eqref{BB3}. 
\end{proof}

Finally, we use the analysis of Method II to also relate Method III to Method I: 
\begin{corollary} {\bf (Method III).}
\label{corollary:minnorm-z}
  Assume that $\{ \pphi_1, \, \pphi_2, \, \ldots, \pphi_n \}$ is an {\em orthonormal  basis} and let 
  \begin{align*}
\zz_{j,\alpha} &= \argmin_{\zz} \left\{ \frac{1}{2} \| \KK \zz - \KK \pphi_j \|_Y^2 + \frac{1}{2} \alpha \| \WW \zz \|_X^2 \right\}, \\
\zz_j^* &= \lim_{\alpha \rightarrow 0} \zz_{j,\alpha}, 
\end{align*}
where $\WW$ is defined in \eqref{A4}. Then 
\begin{align}
\label{BB4}
    \left\| \pphi_j - \zz_j^* \right\|_X \leq \frac{\| \PP \pphi_j \|_X}{\min_{i=1,2,\ldots,n} \| \PP \pphi_i \|_X} \| \pphi_j - \WW^{-1} \xx_j^* \|_X,  
\end{align}
where $\xx_j^*$ is the minimum norm least-squares solution of \eqref{recoveryEq} and $\WW^{-1} \xx_j^*$ can be written in the form \eqref{basicEq}. 
\end{corollary}
\begin{proof}
Let $\yy_j^*$ be the minimum norm least-squares solution of \eqref{BB1}. 
  From the definition \eqref{A4} of $\WW$ and the fact that $\zz_j^* = \WW^{-1} \yy_j^*$ it follows that 
  \[
  \pphi_j -  \zz_j^* = \| \PP \pphi_j \|_X \, \WW^{-1} \pphi_j -  \WW^{-1} \yy_j^*.
  \]
  Therefore, 
  \begin{align*}
  \| \pphi_j -  \zz_j^* \|_X &= \| \PP \pphi_j \|_X \, \left\| \WW^{-1} \pphi_j -  \frac{\WW^{-1} \yy_j^*}{\| \PP \pphi_j \|_X} \right\|_X \\
  &\leq \| \PP \pphi_j \|_X \| \WW^{-1} \| \, \left\| \pphi_j - \frac{\yy_j^*}{\| \PP \pphi_j \|_X} \right\|_X \\
  &= \frac{\| \PP \pphi_j \|_X}{\min_{i=1,2,\ldots,n}  \| \PP \pphi_i \|_X} \left\| \pphi_j -  \frac{\yy_j^*}{\| \PP \pphi_j \|_X} \right\|_X, 
  \end{align*}
  and the results follows from Theorem \ref{theorem:minnorm-y}. 
\end{proof}
Inequality \eqref{BB4} shows that, if $\| \PP \pphi_j \|_X \approx \min_{i=1,2,\ldots,n} \| \PP \pphi_i \|_X$, then Method III can potentially yield results which, in norm sense, is better than Method I. Since $\PP$ is the orthogonal projection onto the orthogonal complement $\nullspace{\KK}^{\perp}$ of the nullspace of $\KK$, this will typically be the case for indexes $j$ corresponding to the basis functions closest to the nullspace of $\KK$. We thus expect Method III to work best for recovering the basis functions closest to the nullspace.  
(However, this 'effect' does not depend on $\min_{i=1,2,\ldots,n} \| \PP \pphi_i \|_X$ being small, only that $\| \PP \pphi_j \|_X \approx \min_{i=1,2,\ldots,n} \| \PP \pphi_i \|_X$.)  

\section{Numerical experiments}
\label{section:numerical_experiments}
We discretized the control $f$ in terms of a rectangular grid with uniformly sized cells $\Omega_1, \, \Omega_2, \ldots, \Omega_n$, and the scaled characteristic functions of these cells were used as basis functions: 
\[
\pphi_i = \frac{1}{\| \mathcal{X}_{\Omega_i}  \|_{L^2(\Omega)}} \mathcal{X}_{\Omega_i}, \quad i = 1, 2, \ldots, n.
\]
Note that this is an $L^2$-orthonormal basis, cf. our theoretical findings in the previous section. 
The state $u$ was discretized with standard first order Lagrange elements, and the stiffness and mass matrices were generated with the FEniCS software system. We imported these matrices into MATLAB and solved the associated optimality systems.

To be in 'exact alignment' with our theoretical findings, except for the use of finite precision arithmetic, we committed the so-called inverse crime in Example 1: The same ($8 \times 8$) grid for the control was used for both solving the inverse problem  and for generating the synthetic Dirichlet observation data $d \in L^2(\partial \Omega)$. Also, the true source equaled one of the basis functions, and thus all the assumptions needed in our theorems were fulfilled. 

In examples 2-7 we avoided inverse crimes by using different grid resolutions for the forward and inverse problems. More specifically, except for Example 2, the synthetic Dirichlet boundary data $d$ in \eqref{eq1}-\eqref{eq2} was generated by solving, on a mesh with $65 \times 65$ grid points, the boundary value problem \eqref{eq2} with $f=f_{\mathrm{true}}$. The data $d$ was then mapped onto a coarser grid with $33 \times 33$ grid points. The inverse problem \eqref{eq1}-\eqref{eq2} was thereafter solved, using $16 \times 16$ and $33 \times 33$ uniform meshes for the control $f$ and the state $u$, respectively. With this procedure, the true sources, employed in the forward simulations, consisted of sums of basis functions with neighbouring supports. (Hence, our analysis, in a strict mathematical sense, can not predict the outcome of these simulations.)  

In Example 2, the forward problem was solved on a non-uniform L-shaped mesh with 548 nodes, before the synthetic boundary data $d$ was mapped onto a coarser grid with 137 nodes. We then solved the inverse problem using the coarse grid for both the control $f$ and the state $u$. 

If not stated otherwise, $\epsilon=10^{-3}$, see \eqref{eq2}, and no noise was added to the synthetic data $d$. 

We observed in panel (b) of Figure \ref{fig:square_ST} that the inverse solution computed with standard Tikhonov regularization fails to recover the true source. Similar results were observed in all the test cases, except when the true source was close to the boundary $\partial \Omega$ of the domain $\Omega$. We will therefore, in most cases, not present further figures generated by applying standard Tikhonov regularization.

\subsection*{Example 1: Simple internal source} 
Figure \ref{fig:square_new} shows the numerical results obtained by solving \eqref{eq1}-\eqref{eq2}, with the regularization operator $\WW$ defined in \eqref{A4}, when the true source is as depicted in panel (a) in Figure \ref{fig:square_ST}. We observe that the location of the true source is recovered rather well by all the three methods, and the results are much better than the inverse solution generated by employing standard Tikhonov regularization, see panel (b) in Figure \ref{fig:square_ST}. Nevertheless, the magnitude of the true source is severely underestimated by all the schemes and the well-known smoothing effect due to 'quadratic regularization' is clearly present. 

\begin{figure}[H]
    \centering
    \begin{subfigure}[b]{0.6\linewidth}        
        \centering
        \includegraphics[width=\linewidth]{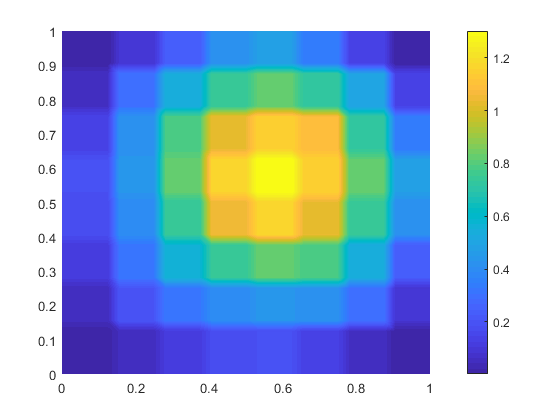}
        \caption{Method I.}
    \end{subfigure}\par
    \begin{subfigure}[b]{0.6\linewidth}        
        \centering
        \includegraphics[width=\linewidth]{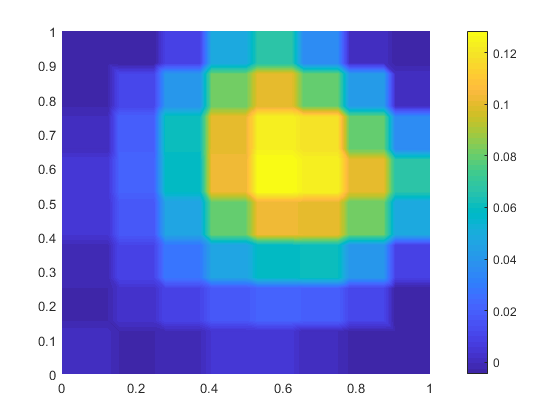}
        \caption{Method II.}
    \end{subfigure}\par
    \begin{subfigure}[b]{0.6\linewidth}        
        \centering
        \includegraphics[width=\linewidth]{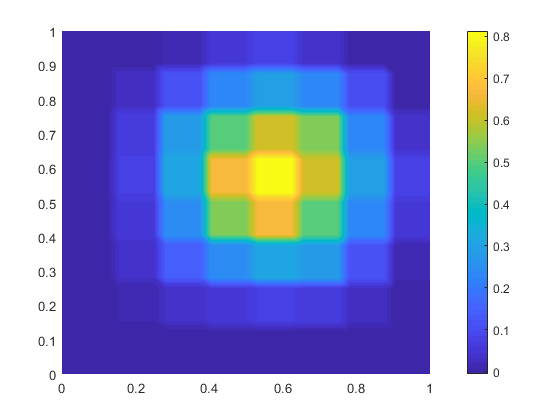}
        \caption{Method III.}
    \end{subfigure}
    \caption{Recovered source, Example 1, with the regularization parameter $\alpha = 10^{-3}$. The true source is depicted in panel (a) in Figure \ref{fig:square_ST}.}
    \label{fig:square_new}
\end{figure}

\subsection*{Example 2: L-shaped geometry}

We will now consider the problem \eqref{eq1}-\eqref{eq2} with an L-shaped domain $\Omega$.  
\begin{figure}[h]
    \centering
    \begin{subfigure}[b]{0.45\linewidth}        
        \centering
        \includegraphics[width=\linewidth]{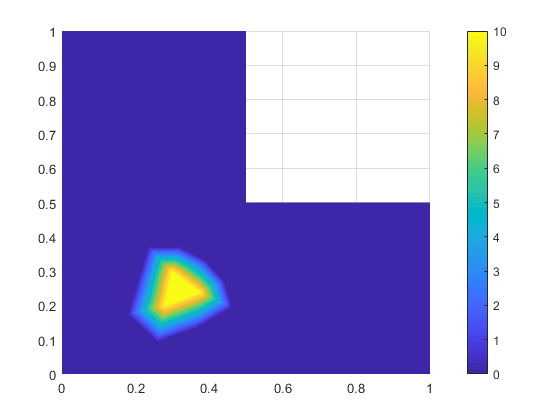}
        \caption{True source}
    \end{subfigure}
    \begin{subfigure}[b]{0.45\linewidth}        
        \centering
        \includegraphics[width=\linewidth]{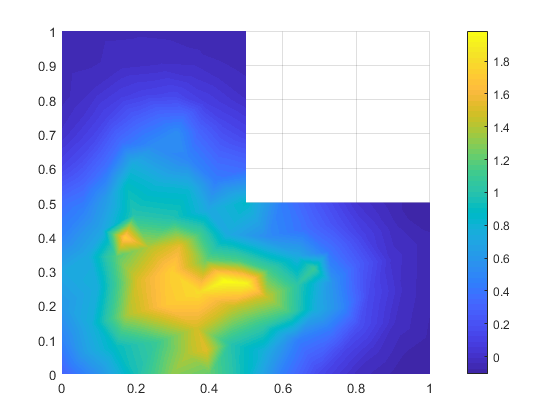}
        \caption{Method I.}
    \end{subfigure}\par
    \begin{subfigure}[b]{0.45\linewidth}        
        \centering
        \includegraphics[width=\linewidth]{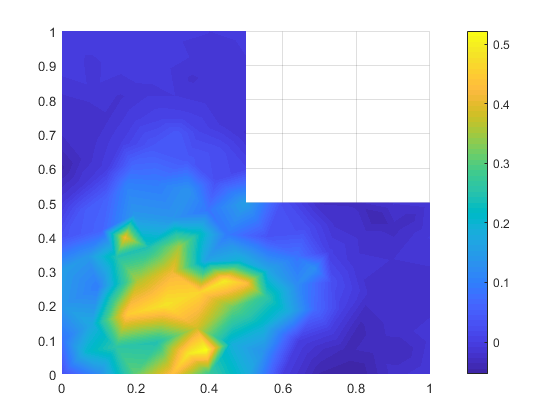}
        \caption{Method II.}
    \end{subfigure}
    \begin{subfigure}[b]{0.45\linewidth}        
    \centering
    \includegraphics[width=\linewidth]{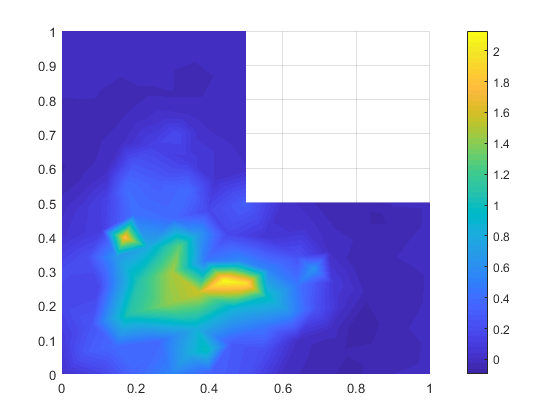}
    \caption{Method III.}
\end{subfigure}
    \caption{L-shaped domain, Example 2. Comparison of the true source and the inverse solutions, using the regularization parameter $\alpha = 10^{-3}$.}
    \label{fig:lshape}
\end{figure}
The location of the true source is identified rather accurately by employing our proposed methods, see panels (b)-(d) in Figure \ref{fig:lshape}. Visually, it appears that Method I produces the best result, whereas particularly Method III generates a solution which is slightly too far to the right. 
In this example, the magnitude of the true source is somewhat better recovered, compared with the results reported in Example 1. 

\subsection*{Example 3: Source at the boundary}
Figure \ref{fig:tikboundary} shows that standard Tikhonov regularization performs somewhat better than the new methods when the  true source is located at the boundary $\partial \Omega$ of the domain $\Omega$, cf. Figure \ref{fig:boundary}. However, all the techniques work rather well in this particular case: The position and the magnitude of the true source is roughly recovered by all the methods.  
\begin{figure}[H]
    \centering
    \begin{subfigure}[b]{0.45\linewidth}        
        \centering
        \includegraphics[width=\linewidth]{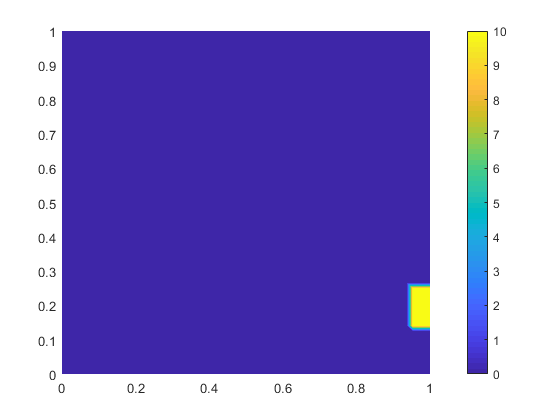}
        \caption{True source}
    \end{subfigure}
    \begin{subfigure}[b]{0.45\linewidth}        
        \centering
        \includegraphics[width=\linewidth]{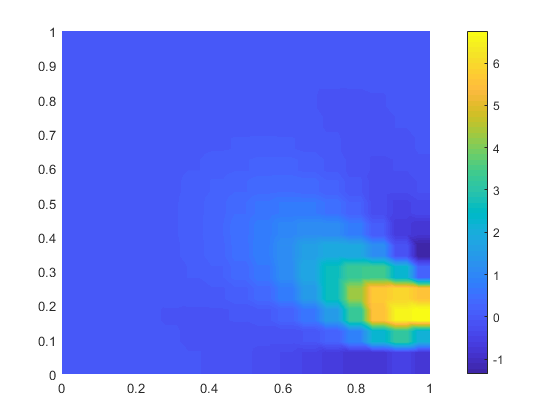}
        \caption{Method I.}
    \end{subfigure}\par
    \begin{subfigure}[b]{0.45\linewidth}        
        \centering
        \includegraphics[width=\linewidth]{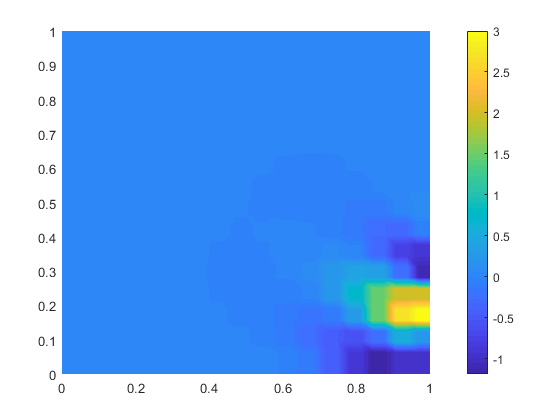}
        \caption{Method II.}
    \end{subfigure}
    \begin{subfigure}[b]{0.45\linewidth}        
    \centering
    \includegraphics[width=\linewidth]{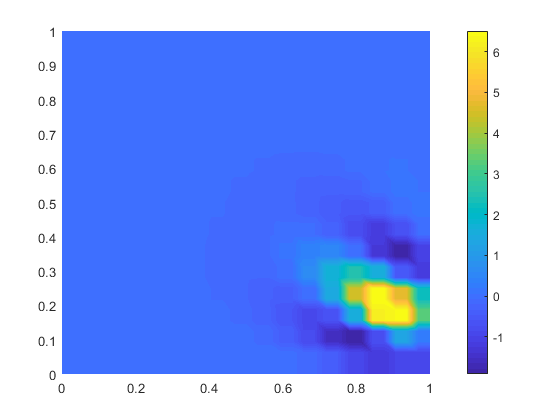}
    \caption{Method III.}
\end{subfigure}
    \caption{Source at the boundary, Example 3. Comparison of the true source and the inverse solutions, using the regularization parameter $\alpha = 10^{-4}$.}
    \label{fig:boundary}
\end{figure}

\begin{figure}[H]
    \centering
    \includegraphics[scale=.5]{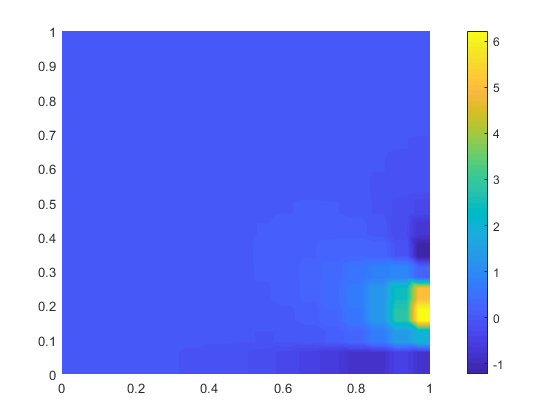}
    \caption{Source at the boundary, Example 3. Inverse solution computed with standard Tikhonov regularization, $\alpha = 10^{-4}$.}
    \label{fig:tikboundary}
\end{figure}

\subsection*{Example 4: Tensor}
The next problem reads:
    \begin{equation} \label{eq:tensor1}
        \min_{(f,u) \in F_h \times H^1(\Omega)} \left\{ \frac{1}{2}\|u-d\|_{L^2(\partial\Omega)}^2 + \frac{1}{2}\alpha\|\WW f\|_{L^2(\Omega)}^2 \right\}
    \end{equation}
    subject to 
    \begin{equation} \label{eq:tensor2}
    \begin{split}
        -\nabla \cdot \sigma \nabla u + \epsilon u &= f \quad \mbox{in } \Omega, \\
        \frac{\partial u}{\partial \nn} &= 0  \quad \mbox{on } \partial \Omega, 
    \end{split}
    \end{equation}
where 
\[
\sigma = \mathrm{diag}(\kappa_1, \, \kappa_2)
\]
is a diagonal and uniformly positive definite $2 \times 2$ matrix (with function entries). In this experiment, the vector field $(\kappa_1, \, \kappa_2)$ is as shown in Figure \ref{fig:tensorfield}.
\begin{figure}[H]
    \centering
    \includegraphics[scale=.7]{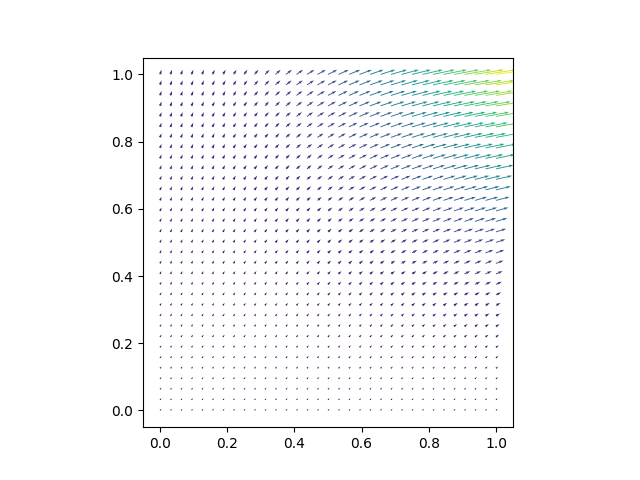}
    \caption{Vector field of $\sigma$.}
    \label{fig:tensorfield}
\end{figure}

Due to the 'overall direction' of the vector field, one could expect that the estimated source would be shifted to the right, and possibly upwards, compared with the true source. 
%
However, from the definition \eqref{A4} of the regularization operator $\WW$, it follows that the tensor $\sigma$ will influence $\WW$, i.e., $\WW = \WW(\sigma)$.  This is in contrast to the standard Tikhonov regularization term, which is unaffected by the presence of a non-constant tensor in the PDE.

Figure \ref{fig:tensor} shows the numerical solutions of \eqref{eq:tensor1} - \eqref{eq:tensor2}, as well as the location of the true source. We observe in panels (b)-(d) that the position of the source is identified rather well by all the three methods. Only a marginal drift to the right can be observed. 

We also applied standard Tikhonov regularization to this problem (figure omitted). The maximum value of the suggested source then occurred at the right part of the boundary, in spite of the fact that the true source is located in the left part of the domain. 
%
%
\begin{figure}[H]
    \centering
    \begin{subfigure}[b]{0.45\linewidth}        
        \centering
        \includegraphics[width=\linewidth]{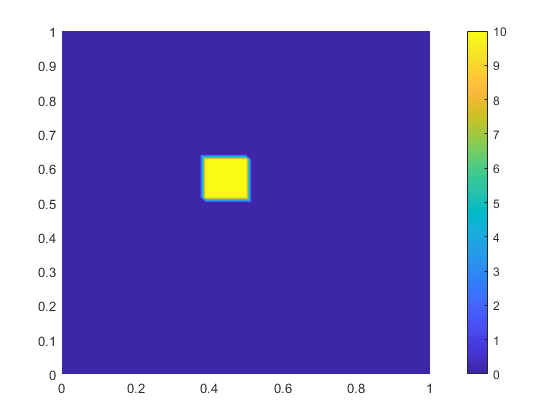}
        \caption{True source}
    \end{subfigure}
    \begin{subfigure}[b]{0.45\linewidth}        
        \centering
        \includegraphics[width=\linewidth]{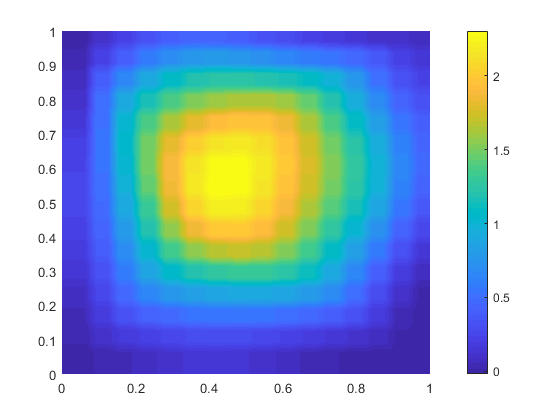}
        \caption{Method I.}
    \end{subfigure}\par
    \begin{subfigure}[b]{0.45\linewidth}        
        \centering
        \includegraphics[width=\linewidth]{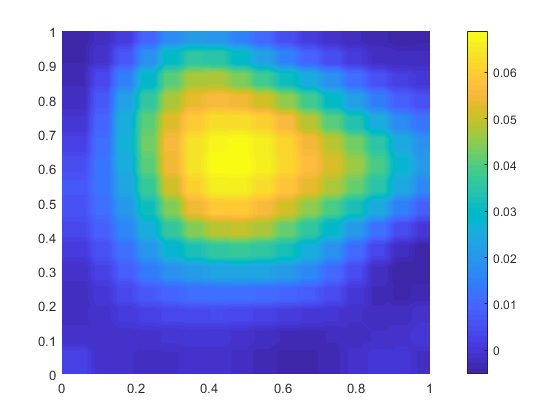}
        \caption{Method II.}
    \end{subfigure}
    \begin{subfigure}[b]{0.45\linewidth}        
    \centering
    \includegraphics[width=\linewidth]{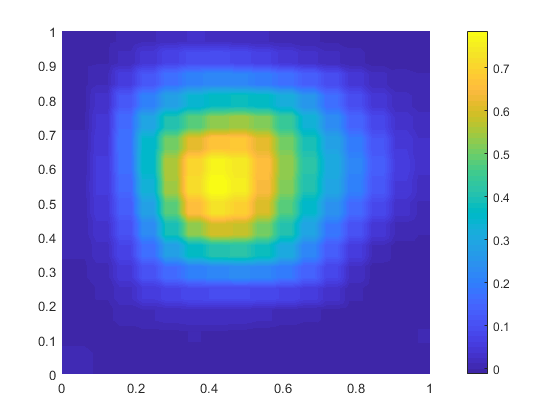}
    \caption{Method III.}
\end{subfigure}
    \caption{State equation with a tensor, Example 4. Comparison of the true source and the inverse solutions, using the regularization parameter $\alpha = 10^{-4}$.}
    \label{fig:tensor}
\end{figure}

\subsection*{Example 5: Multiple sources}
In this subsection we investigate how the new techniques handle multiple sources. Figure \ref{fig:twosources} shows the two-sources case, whereas the results for the three-sources case are displayed in Figure \ref{fig:threesources}.

For the two-sources case, methods II and III rather successfully localize the two regions, see panels (c) and (d) in Figure \ref{fig:twosources}. The two regions are quite clearly distinguishable, even though the inverse solution is much smoother than the true source. In this case, Method I fails to recover the sources, ref. panel (b).

Similarly, for the three-sources case, the 'active' regions are still distinguishable using methods II or III, although two of the three sources are much stronger recovered than the third (bottom right).

\begin{figure}[h]
    \centering
    \begin{subfigure}[b]{0.45\linewidth}        
        \centering
        \includegraphics[width=\linewidth]{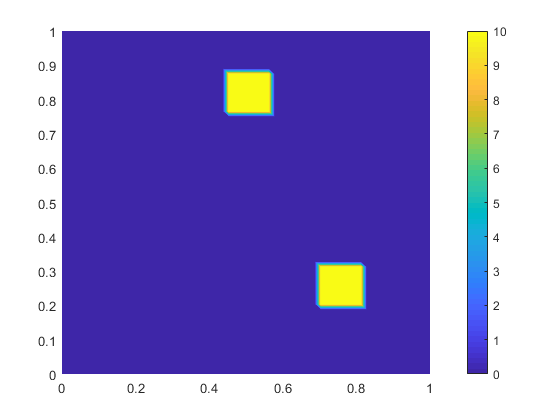}
        \caption{True source}
    \end{subfigure}
    \begin{subfigure}[b]{0.45\linewidth}        
        \centering
        \includegraphics[width=\linewidth]{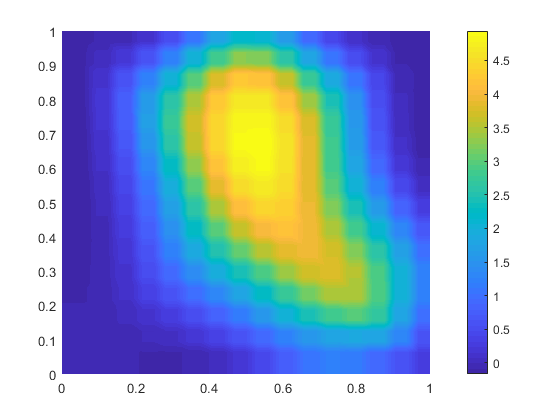}
        \caption{Method I.}
    \end{subfigure}\par
    \begin{subfigure}[b]{0.45\linewidth}        
        \centering
        \includegraphics[width=\linewidth]{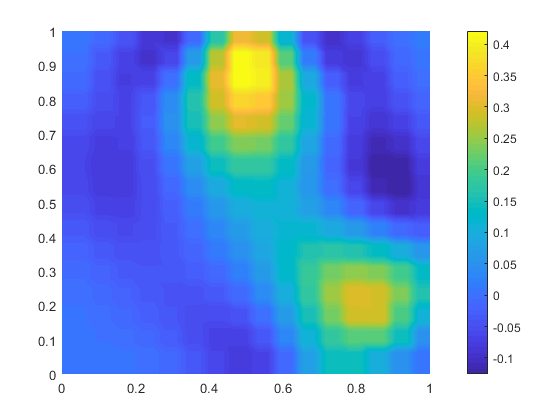}
        \caption{Method II.}
    \end{subfigure}
    \begin{subfigure}[b]{0.45\linewidth}        
    \centering
    \includegraphics[width=\linewidth]{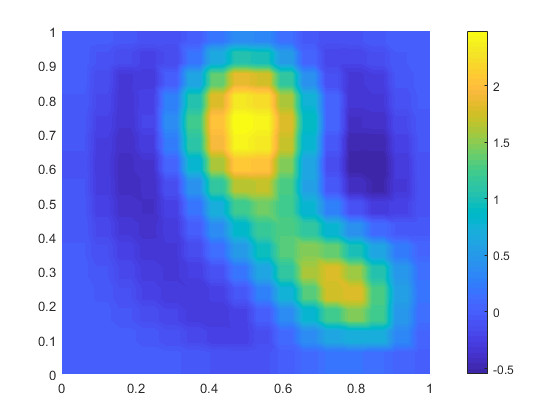}
    \caption{Method III.}
\end{subfigure}
    \caption{Two disjoint sources, Example 5. Comparison of the true sources and the inverse solutions, using the regularization parameter $\alpha = 10^{-3}$.}
    \label{fig:twosources}
\end{figure}
\begin{figure}[h]
    \centering
    \begin{subfigure}[b]{0.45\linewidth}        
        \centering
        \includegraphics[width=\linewidth]{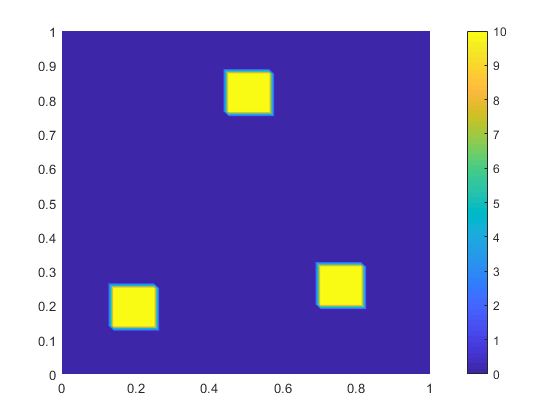}
        \caption{True source}
    \end{subfigure}
    \begin{subfigure}[b]{0.45\linewidth}        
        \centering
        \includegraphics[width=\linewidth]{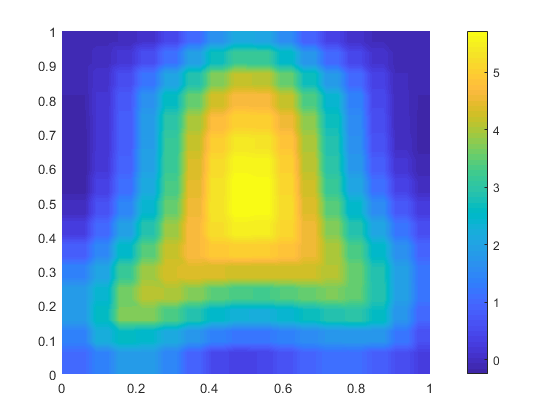}
        \caption{Method I.}
    \end{subfigure}\par
    \begin{subfigure}[b]{0.45\linewidth}        
        \centering
        \includegraphics[width=\linewidth]{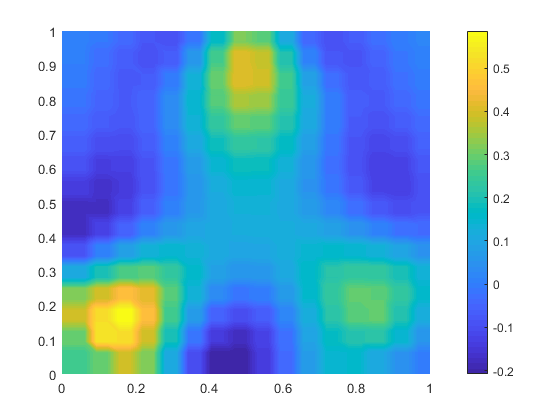}
        \caption{Method II.}
    \end{subfigure}
    \begin{subfigure}[b]{0.45\linewidth}        
    \centering
    \includegraphics[width=\linewidth]{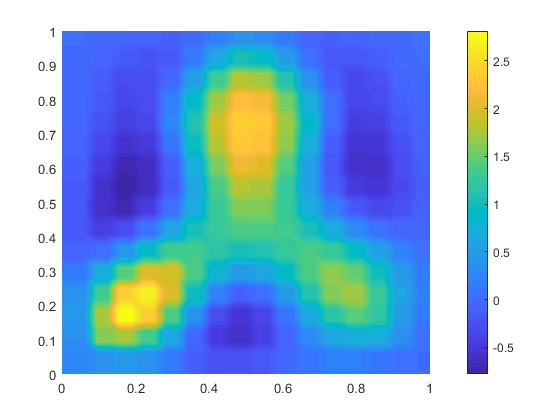}
    \caption{Method III.}
\end{subfigure}
    \caption{Three disjoint sources, Example 5. Comparison of the true sources and the inverse solutions, using the regularization parameter $\alpha = 10^{-3}$.}
    \label{fig:threesources}
\end{figure}

\subsection*{Example 6: Noisy data} 
Next, we explore the performance of the new methodology for noisy observation data $d$.
More specifically, we now assume that only an approximation $d^\delta$ of $d$ is known: 
\begin{equation}
    d^\delta(x) = d(x) + \delta\rho(x),
\end{equation}
where $\rho(x)$ is a normally distributed stochastic variable with zero mean and standard deviation equal to 1. The scalar $\delta$ is  
\begin{equation*}
    \delta = \kappa \left(\max_{x\in\partial\Omega} d(x) - \min_{x\in\partial\Omega} d(x)\right),
\end{equation*}
and we define the noise level to be the standard deviation of $\delta \rho$ relatively to the range of the data $d$, i.e., 
\[
\frac{\sigma(\delta \rho)}{\max_{x\in\partial\Omega} d(x) - \min_{x\in\partial\Omega} d(x)} = \frac{\delta}{\max_{x\in\partial\Omega} d(x) - \min_{x\in\partial\Omega} d(x)} = \kappa.
\]

The regularization parameter was chosen according to  Mozorov's discrepancy principle, i.e., when
\begin{equation*}
    \|d^\delta - d\|_{L^2(\partial\Omega)} = \gamma,
\end{equation*}
we chose $\alpha$ such that 
\begin{equation*}
    \|\KK f_\alpha^\delta - d^\delta\|_{L^2(\partial\Omega)} = \gamma.
\end{equation*}

\begin{figure}[h]
    \centering
    \begin{subfigure}[b]{0.45\linewidth}        
        \centering
        \includegraphics[width=\linewidth]{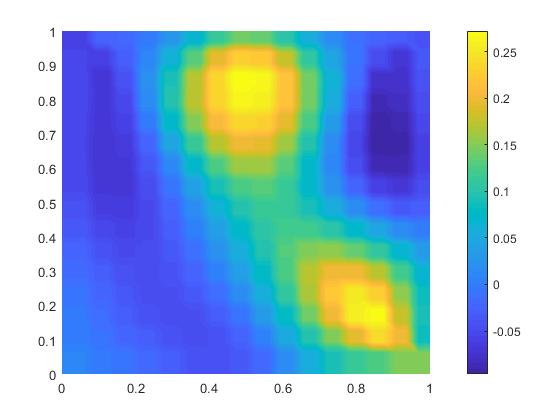}
        \caption{Method II, 5\% noise.}
    \end{subfigure}
    \begin{subfigure}[b]{0.45\linewidth}        
    \centering
    \includegraphics[width=\linewidth]{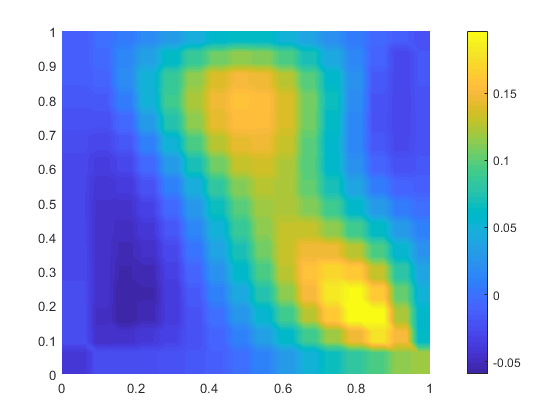}
    \caption{Method II, 20\% noise.}
    \end{subfigure}\par
    \begin{subfigure}[b]{0.45\linewidth}        
        \centering
        \includegraphics[width=\linewidth]{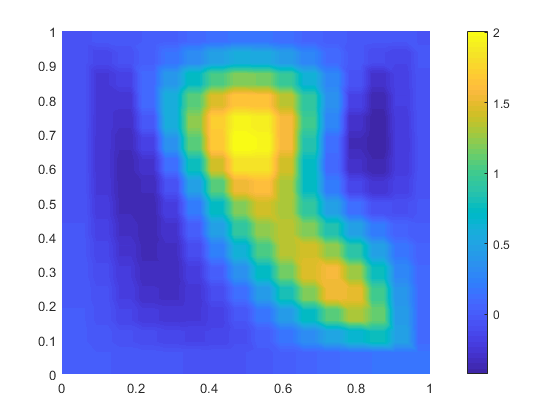}
        \caption{Method III, 5\% noise.}
    \end{subfigure}
    \begin{subfigure}[b]{0.45\linewidth}        
        \centering
        \includegraphics[width=\linewidth]{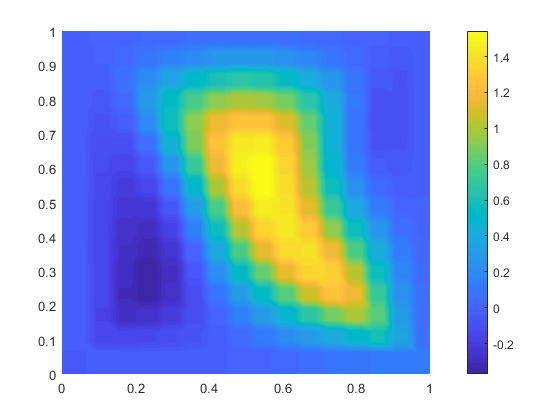}
        \caption{Method III, 20\% noise.}
    \end{subfigure}    
    \caption{Example 6, $5 \%$ and $20\%$ noise. The true source is shown in panel (a) in Figure \ref{fig:twosources}.}
    \label{fig:noisy}
\end{figure}

We used the same true source as in Example 5, see panel (a) in Figure \ref{fig:twosources}. 
Figure \ref{fig:noisy} shows plots of the numerical simulations with Method II and Method III when the noise level is $5 \%$ and $20 \%$. Both methods produce solutions which indicate the regions of the true sources rather well when the noise level is $5 \%$. In the case of $20 \%$ noise, only Method II is able to distinguish the two regions. 
Since Method I failed also for the noise-free case, we do not present the simulations with noisy observation data for this method. 

\subsection*{\textcolor{black}{Example 7: Inhomogeneous Helmholtz equation}}

\textcolor{black}{Finally, we consider two examples where $\epsilon < 0$, i.e., the PDE in \eqref{eq2} becomes the inhomogeneous Helmholtz equation. As mentioned in section \ref{section:motivation}, maximum principles for functions satisfying the Helmholtz equation is not readily available. Consequently, we can not use the argument presented in section \ref{section:motivation} to assert whether Tikhonov regularization will fail to identify internal sources.}

\textcolor{black}{The results are displayed in Figure \ref{fig:helmholtz1} and Figure \ref{fig:helmholtz2} for $\epsilon = -1$ and $\epsilon = -100$, respectively. The true source is as in Example 1, see Figure \ref{fig:square_ST_a}. For both choices of $\epsilon$, we observe that methods I, II and III are able to locate the position of the true source, see panels (b)-(d), whereas employing standard Tikhonov regularization is unsuccessful for $\epsilon = -1$, but works well when $\epsilon = -100$, cf. panel (a) in Figure \ref{fig:helmholtz2}. It is not clear to us why the use of standard Tikhonov regularization is sufficient to locate the interior source when $\epsilon = -100$.}



\begin{figure}[h]
    \centering
    \begin{subfigure}[b]{0.45\linewidth}        
        \centering
        \includegraphics[width=\linewidth]{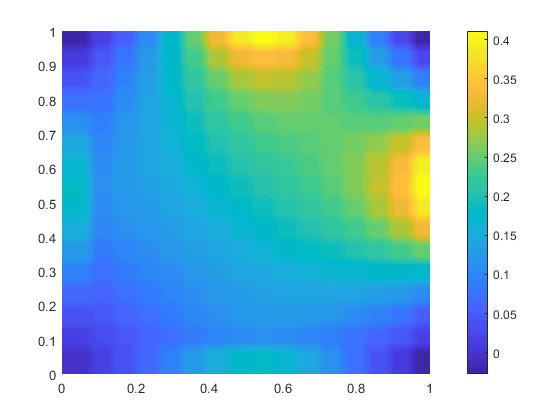}
        \caption{Standard Tikhonov regularization}
    \end{subfigure}
    \begin{subfigure}[b]{0.45\linewidth}        
        \centering
        \includegraphics[width=\linewidth]{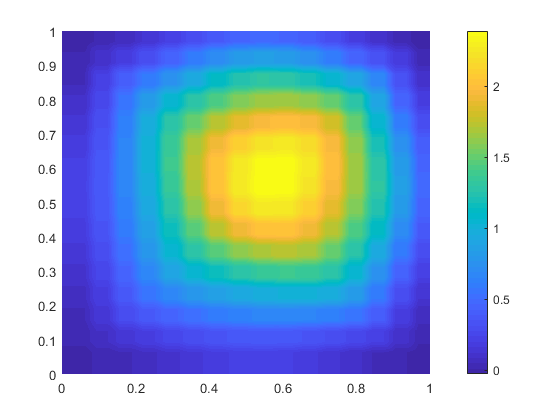}
        \caption{Method I.}
    \end{subfigure}\par
    \begin{subfigure}[b]{0.45\linewidth}        
        \centering
        \includegraphics[width=\linewidth]{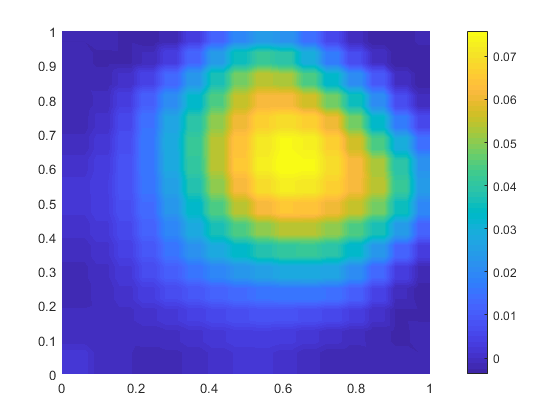}
        \caption{Method II.}
    \end{subfigure}
    \begin{subfigure}[b]{0.45\linewidth}        
    \centering
    \includegraphics[width=\linewidth]{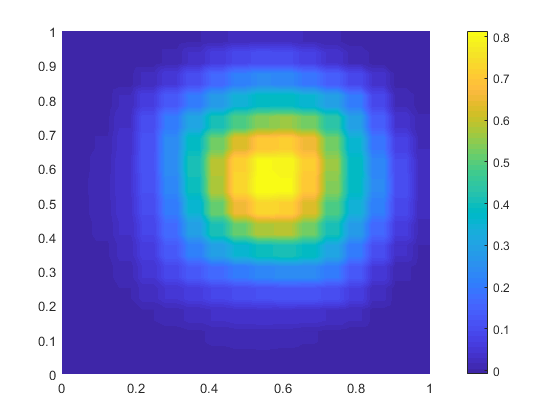}
    \caption{Method III.}
\end{subfigure}
    \caption{Inhomogeneous Helmholtz equation with $\epsilon = -1$. Comparison of the inverse solutions, using the regularization parameter $\alpha = 10^{-3}$. The true source is displayed in Figure \ref{fig:square_ST_a}.}
    \label{fig:helmholtz1}
\end{figure}

\begin{figure}[h]
    \centering
    \begin{subfigure}[b]{0.45\linewidth}        
        \centering
        \includegraphics[width=\linewidth]{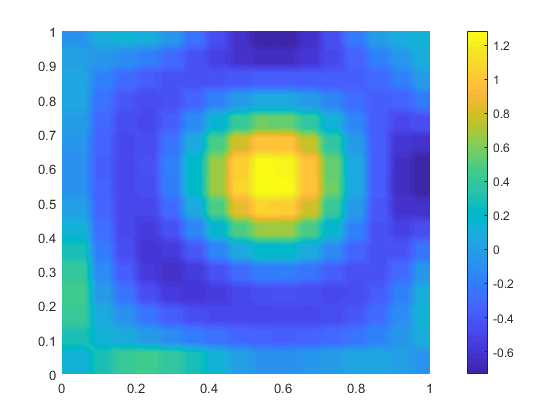}
        \caption{Standard Tikhonov regularization}
    \end{subfigure}
    \begin{subfigure}[b]{0.45\linewidth}        
        \centering
        \includegraphics[width=\linewidth]{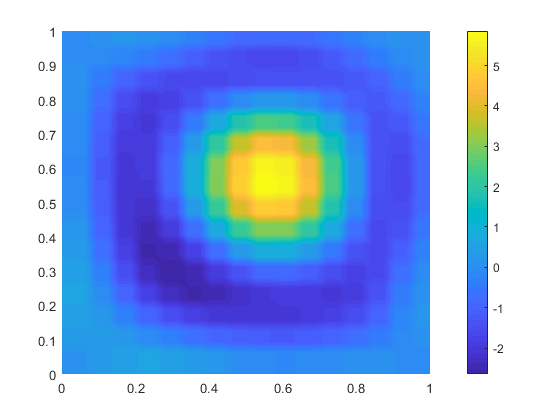}
        \caption{Method I.}
    \end{subfigure}\par
    \begin{subfigure}[b]{0.45\linewidth}        
        \centering
        \includegraphics[width=\linewidth]{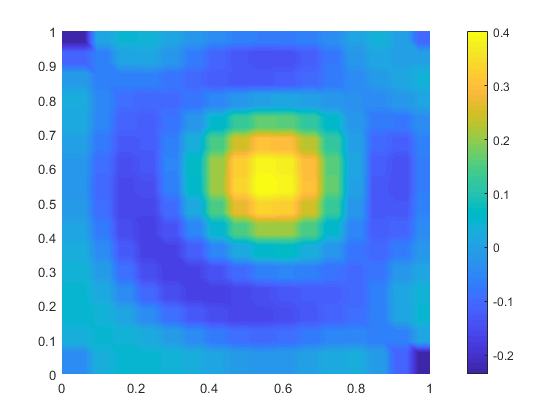}
        \caption{Method II.}
    \end{subfigure}
    \begin{subfigure}[b]{0.45\linewidth}        
    \centering
    \includegraphics[width=\linewidth]{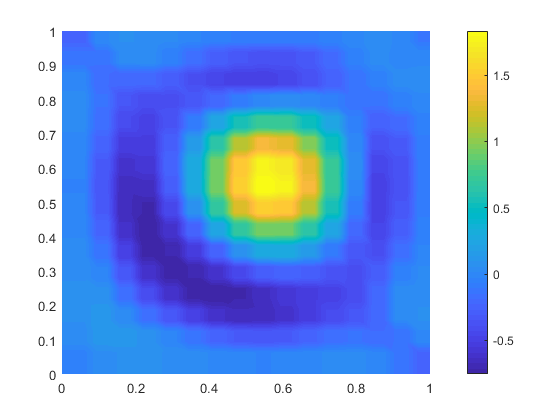}
    \caption{Method III.}
\end{subfigure}
    \caption{Inhomogeneous Helmholtz equation with $\epsilon = -100$. Comparison of the inverse solutions, using the regularization parameter $\alpha = 10^{-3}$. The true source is displayed in Figure \ref{fig:square_ST_a}.}
    \label{fig:helmholtz2}
\end{figure}

\section{\textcolor{black}{Remarks and an open problem}}
\label{section:discussion}
The methods developed in this paper can recover a single 
source positioned anywhere in $\Omega$, and the new schemes are defined in terms of standard quadratic optimization problems. They are thus simple to implement. Our results are supported by rigorous mathematical analysis and by numerical experiments. 

\textcolor{black}{Both our theoretical investigations and experiments show that if the true sources are very close to the boundary $\partial \Omega$, then standard Tikhonov regularization is preferable. On the other hand, methods I, II or III should be applied if one wants to identify internal sources. Moreover, we know that, for a particular simple class of problems, a scaled version of Method II yields better approximations than Method I, see Theorem \ref{theorem:minnorm-y}. Corollary \ref{corollary:minnorm-z} expresses that a similar result, though somewhat weaker, also holds for Method III.}

The examples, presented in the previous section, indicate that \textcolor{black}{methods II and III} also can identify several isolated sources. Nevertheless, we have not presented any mathematical analysis for such cases. A more thorough understanding of this is an open problem. 

\appendix

\section{Alternative motivation for $\WW$}
\label{alternative-motivation}
Consider a regularization operator in the form 
\[
\WW \phi_i = w_i \pphi_i, \, w_i \in \mathbb{R} \setminus \{0\}, \, i=1,2,\ldots,n.
\]
The minimum norm least-squares solution $\yy^*$ of \eqref{A3} belongs to $\nullspace{\KK \WW^{-1}}^{\perp}$: 
\[
(\yy^*, \WW q)_X = 0 \quad \mbox{for all } q \in \nullspace{\KK} 
\]
because  
\[
q \in \nullspace{\KK} \iff \WW q \in \nullspace{\KK \WW^{-1}}. 
\]
Since $\WW$ is self-adjoint, it follows that 
\[
(\WW \yy^*, q)_X = 0 \quad \mbox{for all } q \in \nullspace{\KK},  
\]
and we conclude that 
\begin{equation}
\label{app1}
\WW \yy^* \in \nullspace{\KK}^{\perp}. 
\end{equation}

Assume that one wants to use the minimum norm least-squares solution $\yy^*$ of \eqref{A3} to approximately recover $\pphi_i$. That is, we want to choose $\WW$ such that the minimum norm least-squares solution $\yy^* \approx \pphi_i$ is possible. Choosing $\yy^* \approx \pphi_i$ in  \eqref{app1} yields that $\WW \pphi_i = w_i \pphi_i$ approximately must belong to $\nullspace{\KK}^{\perp}$. We propose to achieve this by requiring that $w_i \pphi_i$ is as close as possible to $\frac{\PP \pphi_i}{\| \PP \pphi_i \|_X} \in \nullspace{\KK}^{\perp}$, where $\PP$ is the orthogonal projection \eqref{A3.1}. This suggests the choice 
\[
w_i = \argmin_{c} \left\| c \pphi_i - \frac{\PP \pphi_i}{\| \PP \pphi_i \|_X}   \right\|_X^2 = \| \PP \pphi_i \|_X.  
\]
 In other words, we choose $w_i$ such that $w_i \pphi_i$ gets as close as possible to the normalized best approximation of $\pphi_i$ in $\nullspace{\KK}^{\perp}$. 


\bibliographystyle{plain}
\bibliography{references}

\begin{thebibliography}{10}

\bibitem{Abdel15}
B.~Abdelaziz, A.~{El Badia}, and A.~{El Hajj}.
\newblock Direct algorithms for solving some inverse source problems in 2{D}
  elliptic equations.
\newblock {\em Inverse Problems}, 31(10):105002, 2015.

\bibitem{Alves04}
C.~J.~S. Alves, J.~B. Abdallah, and M.~Jaoua.
\newblock Recovery of cracks using a point-source reciprocity gap function.
\newblock {\em Inverse Problems in Science and Engineering}, 12(5):519--534,
  2004.

\bibitem{bail01}
S.~Baillet, J.~C. Mosher, and R.~M. Leahy.
\newblock Electromagnetic brain mapping.
\newblock {\em IEEE Signal Processing Magazine}, 18(6):14--30, 2001.

\bibitem{babda09}
A.~Ben~Abda, F.~Ben~Hassen, J.~Leblond, and M.~Mahjoub.
\newblock Sources recovery from boundary data: A model related to
  electroencephalography.
\newblock {\em Mathematical and Computer Modelling}, 49:2213--2223, 2009.

\bibitem{cheng15}
X.~Cheng, R.~Gong, and W.~Han.
\newblock A new {K}ohn-{V}ogelius type formulation for inverse source problems.
\newblock {\em Inverse Problems and Imaging}, 9(4):1051--1067, 2015.

\bibitem{Bad98}
A.~El~Badia and T.~Ha-Duong.
\newblock Some remarks on the problem of source identification from boundary
  measurements.
\newblock {\em Inverse Problems}, 14:883–891, 1998.

\bibitem{Bad00}
A.~El~Badia and T.~Ha-Duong.
\newblock An inverse source problem in potential analysis.
\newblock {\em Inverse Problems}, 16:651--663, 2000.

\bibitem{elul72}
R.~Elul.
\newblock The genesis of the {EEG}.
\newblock {\em International review of neurobiology}, 15:227--272, 1972.

\bibitem{BEng96}
H.~W. Engl, M.~Hanke, and A.~Neubauer.
\newblock {\em Regularization of Inverse Problems}.
\newblock Kluwer Academic Publishers, 1996.

\bibitem{Han11}
M.~Hanke and W.~Rundell.
\newblock On rational approximation methods for inverse source problems.
\newblock {\em Inverse Problems and Imaging}, 5(1):185–202, 2011.

\bibitem{Het96}
F.~Hettlich and W.~Rundell.
\newblock Iterative methods for the reconstruction of an inverse potential
  problem.
\newblock {\em Inverse Problems}, 12:251--266, 1996.

\bibitem{hinze19}
M.~Hinze, B.~Hofmann, and T.~N.~T. Quyen.
\newblock A regularization approach for an inverse source problem in elliptic
  systems from single {C}auchy data.
\newblock {\em Numerical Functional Analysis and Optimization},
  40(9):1080--1112, 2019.

\bibitem{BIsa05}
V.~Isakov.
\newblock {\em Inverse Problems for Partial Differential Equations}.
\newblock Springer-Verlag, 2005.

\bibitem{kun94}
K.~Kunisch and X.~Pan.
\newblock Estimation of interfaces from boundary measurements.
\newblock {\em SIAM J. Control Optim.}, 32(6):1643–1674, 1994.

\bibitem{Nie13a}
B.~F. Nielsen, M.~Lysaker, and P.~Gr{\o}ttum.
\newblock Computing ischemic regions in the heart with the bidomain model;
  first steps towards validation.
\newblock {\em IEEE Transactions on Medical Imaging}, 32(6):1085--1096, 2013.

\bibitem{BRen93}
M.~Renardy and R.~C. Rogers.
\newblock {\em An Introduction to Partial Differential Equations}.
\newblock Springer-Verlag, 1993.

\bibitem{ring95}
W.~Ring.
\newblock Identification of a core from boundary data.
\newblock {\em SIAM Journal on Applied Mathematics}, 55(3):677--706, 1995.

\bibitem{song12}
S.~J. Song and J.~G. Huang.
\newblock Solving an inverse problem from bioluminescence tomography by
  minimizing an energy-like functional.
\newblock {\em J. Comput. Anal. Appl.}, 14:544--558, 2012.

\bibitem{Wan13}
D.~Wang, R.~M. Kirby, R.~S. MacLeod, and C.~R. Johnson.
\newblock Inverse electrocardiographic source localization of ischemia: {A}n
  optimization framework and finite element solution.
\newblock {\em Journal of Computational Physics}, (250):403--424, 2013.

\bibitem{Wan17}
X.~Wang, Y.~Guo, D.~Zhang, and H.~Liu.
\newblock Fourier method for recovering acoustic sources from multi-frequency
  far-field data.
\newblock {\em Inverse Problems}, 33(3), 2017.

\bibitem{Zha18}
D.~Zhang, Y.~Guo, J.~Li, and H.~Liu.
\newblock Retrieval of acoustic sources from multi-frequency phaseless data.
\newblock {\em Inverse Problems}, 34(9), 2018.

\bibitem{Zha19}
D.~Zhang, Y.~Guo, J.~Li, and H.~Liu.
\newblock Locating multiple multipolar acoustic sources using the direct
  sampling method.
\newblock {\em Communications in Computational Physics}, 25(5):1328--1356,
  2019.

\end{thebibliography}

\end{document}